\documentclass[11pt,a4paper]{article}

\usepackage{a4wide}
\usepackage[latin1]{inputenc}

\usepackage{amsfonts}
\usepackage{amsmath}
\usepackage{amssymb}
\usepackage{amsthm}

\usepackage{verbatim} % adds environment for commenting out blocks of text & for better verbatim

\usepackage{color}

\newtheorem{Theorem}{Theorem}
\newtheorem{Lemma}[Theorem]{Lemma}

\newtheorem{Proposition}[Theorem]{Proposition}
 \theoremstyle{Definition}
\newtheorem{Definition}[Theorem]{Definition}

\def\cc#1{\mathcal{#1}}

\newcommand{\N}{{\mathbb N}}

\newcommand{\floor}[1]{\lfloor #1 \rfloor }

\newcommand{\D}[1]{D^{\tiny \text{\it #1}}}

\newcommand{\cf}{\!\text{\it cf}}

\newcommand{\start}{\text{\it start}}
\renewcommand{\ldots}{ ... }

\begin{document}

\title{On absolutely normal and continued fraction normal numbers}

\author{\begin{tabular}{lcr}
Ver\'onica Becher & \hspace*{1cm}&Sergio A. Yuhjtman
\end{tabular}}

\maketitle

\begin{abstract}
We give a construction of a real number that is normal to all integer bases and  
continued fraction normal. The computation of the first $n$ digits of  its 
continued fraction expansion performs in the order of $n^4$ mathematical operations.   
The construction works by defining successive refinements of appropriate subintervals 
to achieve, in the limit, simple normality to all integer bases and continued fraction normality. 
The main difficulty is to  control the length of these subintervals. 
To achieve this we adapt  and combine known metric theorems 
on continued fractions  and on expansions in integers bases.
\end{abstract}

\bigskip
\bigskip

For a real number  $x$ in the unit interval, 
the continued fraction expansion of $x$
is a sequence  of positive integers  $ a_1, a_2, \ldots $,
 such that
\[
x= %a_0  + 
\cfrac{1}{a_1 + \cfrac{1}{a_2 + \cfrac{1}{\ddots \, + \cfrac{1}{a_n + \cfrac{1}{\ddots} }   }}}
\]
% where $a_0$ is an integer  for each $i>0$, $a_i$ is positive.
%If $x$ is rational, its continued fraction is finite and  the last digit  is greater than $1$.
A real number is continued fraction normal 
if  every block of  integers occurs in the continued fraction expansion 
with  the asymptotic frequency determined by  the Gauss measure.
An application of Birkhoff's Ergodic Theorem 
proves that almost all --with respect to Lebesgue measure--
real numbers  have a  normal continued fraction expansion.

For each real number $x$ in the unit interval, 
its expansion   in an integer  base $b$ greater than or equal to~$2$
 is a sequence of integers $a_1, a_2 \ldots$, where    $0\leq a_i<b$ for every $i$,
such that   
\[
x = \sum_{i=1}^{\infty} a_i b^{-i}.
\] 
% where for each $i>0$, $0\leq a_i<b$ and 
We require that $a_i<b-1$ infinitely many times to  ensure 
that every rational number has a unique representation.
%A real number  $x$  is  normal to a given base if 
%every   block of digits of the same size occurs in the expansion of $x$ in 
%that base  with the same asymptotic frequency.
A real number  $x$  is  simply normal to a given base $b$ if 
every    digit occurs in the $b$-ary expansion of $x$  with the same asymptotic frequency.
Normality to base $b$ is defined as simple normality to $b, b^2, b^3, \ldots$, all the powers of $b$. 
In 1940 Pillai proved that this formulation of  normality  
is equivalent to Borel's original definition in~\cite{Borel1909}
(see~\cite[Theorem~4.2]{Bugeaud2012}).
Absolute normality is defined as  normality to every integer base greater than or equal to~$2$,
hence, as  simple normality to every integer base greater than or equal to~$2$.
Borel  showed that  almost all (with respect to Lebesgue measure) real numbers
 are absolutely normal.

\pagebreak
Thus, almost all real numbers are absolutely normal and continued fraction normal.
\linebreak
Here we prove the following:

\begin{Theorem}\label{thm}
There is an algorithm that computes a number that is absolutely normal and continued fraction normal.
The computation of the first~$n$ digits of the continued fraction expansion
performs a number of mathematical operations  that is in~$O(n^4)$. 
\end{Theorem}

{\bf On the problem of constructing a  number satisfying the two forms of normality.}
The problem  appeared explicitly  in the literature first  in~\cite{Queffelec2006}  
and then  in~\cite[Problem~10.49]{Bugeaud2012}. 
Recently Scheerer~\cite{Scheerer2017b} gave an algorithm that yields one such number
with doubly exponential computational complexity: 
the computation of the first $n$ digits of the continued fraction expansion
performs  doubly exponential in~$n$ many  mathematical operations.
Thus, as any other algorithm  with exponential complexity, Scheerer's algorithm
 can not run in human time. 
In contrast,  algorithms with polynomial complexity can.
To prove Theorem~\ref{thm} we give an algorithm that
can be implemented as an efficient computer program 
that outputs one digit after the other.
Our technique  elaborates on the algorithm  given by Becher, Heiber and Slaman~\cite{poly} which
 has just above quadratic complexity.
Madritsch, Scheerer and Tichy~\cite{MadritschScheererTichy2017}
also elaborated on the algorithm given in~\cite{poly}, but in a very different way,   
and they use it to compute  absolutely normal Pisot numbers efficiently.
Finally we remark that if  in our algorithm
we skip the treatment of  integer bases then
we obtain  a continued fraction normal number 
(with no guarantee of normality to integer bases).
Such an algorithm  differs substantially  from all the previously known  
constructions of continued fraction normal 
numbers~\cite{Pos1957,Adler1985,Bugeaud2012,Madritsch2016,Scheerer2017b}. 
\medskip

{\bf About the proof of Theorem \ref{thm}}.
Our algorithm  works incrementally to define, 
in the limit of the computation,  a real number $x$ in the unit interval.
The construction works by defining successive refinements 
of appropriate subintervals to achieve, in the limit, simple normality 
to all integer bases and continued fraction normality. 
The choice of each subinterval determines further digits in the expansion of $x$ in  integer bases and
in its continued fraction.  
We require that the choice %and that it takes few operations
 contributes to the two forms of normality
 but without revisiting the previous digits.
For this we need to control, at each step of the construction, 
the lengths of the new subintervals, as they  should be not too small.

We say that an interval $I$ is \cf-ary if 
 there is a finite continued fraction  $[a_1,\dotsc, a_n]$
such that the interval $I$ is equal to the set of  all  the numbers whose first $n$ 
digits of their continued fraction expansion  are $a_1,\dotsc, a_n$.
And we say that an interval $J$ is $b$-ary for some integer $b$ greater than or equal to $2$
if there is a finite block $d_1,\dotsc, d_n$ of digits  between $0$ and $b-1$ 
such that  $J$ is equal to the set of real numbers 
whose first $n$ digits of their $b$-ary expansion are equal to $d_1,\dotsc, d_n$.
The set of $b$-ary intervals determined by $n$ digits between $0$ and $b-1$
is a  partition of the unit interval in finitely many parts of equal length.
The set of \cf-ary intervals determined by $n$ digits  
also form a partition of the unit interval,
but in infinite parts of different length.
Our construction rests on the fact that the \cf-ary intervals 
determined by any $n$ digits  have,  nevertheless, an expected approximate length.
This is because the distribution 
of the logarithm of the convergent of  finite continued fractions of 
$n$ digits obeys a Gaussian law.  
Ibragimov~\cite{Ibragimov1961}  was the first to establish this theorem.
Morita~\cite[Theorem~8.1]{Morita1994}  and  Vall\'ee~\cite[Theorem~9]{Vallee1997}
obtained the same result with an optimal error term.
We use this optimal version in Lemma~\ref{lemma:cfgrande} to 
 guarantee the existence of enough \cf-ary subintervals
having the desired relative length with respect to the previously considered \cf-ary interval.
The control of the length of the $b$-ary subintervals is much simpler.
Proposition~\ref{prop:ladrillo} gives the needed estimate for the relative size of a 
$b$-ary subinterval of any given interval. 

To achieve continued fraction normality  
we need to bound the measure of the set 
inside any given \cf-ary  interval $I$
having too many or too few 
occurrences of a given block of digits in their continued fraction expansions.
This is essentially a result on large deviations proved  by 
 Kifer, Peres and Weiss~\cite[Corollary 3.2]{KiferPeresWeiss2001}  
but conditioned on the first digits that determine the interval~$I$.
We establish this result in Lemma~\ref{lemma:relative}.
To achieve normality in each integer base~$b$
we use Hardy and Wright's
estimate~\cite[Proof of Theorem 148]{HardyWright2008} 
for the number of blocks having 
too many or too few occurrences of a given digit, 
as stated in Lemma~\ref{lemma:BHS}.
At each step of the computation the algorithm determines 
a  finite extension of the continued fraction expansion of $x$, 
as well as a  finite extension of the expansion of $x$ in 
each of the finitely many integer bases designated for that step.
To obtain absolute normality 
the set of designated integer bases  increases with the  step number, 
and in the limit it consists of all integers greater than or equal to $2$.
To obtain continued fraction normality,
at each step the algorithm considers the occurrences of blocks from certain finite 
collection. The set of designated blocks increases monotonically in the step number 
and in the limit consists of all blocks of all positive integers.
\smallskip

{\bf Organisation of the paper}.
We devote Section \ref{sec:tools} to  the definitions and 
the  tools to be used  in the proof of Theorem~\ref{thm}.
We first present a non-recursive formulation of the convergent 
of a finite continued fraction 
and we use it  in Lemma~\ref{lemma:relative} to obtain 
convenient upper and lower  bounds of the  length of any  
\cf-ary subinterval  of a given \cf-ary interval.
These upper and lower bounds
propagate along most of the results of this work.
We give  aforementioned key  Lemmas \ref{lemma:cfgrande},
\ref{lemma:KPWc} and \ref{lemma:BHS},
as well as the material  to deal  with the  discrepancy associated to
continued fraction expansions and $b$-ary expansions,
see Lemmas \ref{lemma:Dcf} and \ref{lemma:D}.
In Section \ref{sec:proof}  we give the actual proof of Theorem~\ref{thm}. 
We present the algorithm,  we prove its correctness and we 
estimate the number of mathematical operations needed to 
compute the first $n$ digits of the continued 
fraction expansion of the  number defined by our algorithm.
The algorithm and its correctness are based on  Lemma \ref{lemma:main},
which is the main lemma of the paper.

\section{Needed definitions and  lemmas}\label{sec:tools}

{\bf Notation}. As usual we write $\N$ 
to denote the set of positive integers and  $\N^k$ to denote the  set of $k$ tuples of positive integers.
For a finite set $S$, $\#S$ is its cardinality.
For  an infinite set  $S$ of real numbers, $|S|$ is its Lebesgue measure;
hence, when $S$ is an interval, $|S|$ is its length.
We  use standard notation for the asymptotic behaviour of functions. 
We say that a function $g(x)$ is  $O(f(x))$ if  there are constants $x_0$ and $c$ 
such that for every $x\geq x_0$, $|g(x)| < c \cdot |f(x)|$.
We write $\log$ to denote the logarithm in base~$e$.

\subsection{\cf-ary intervals}

%We write $x=[a_0; a_1, a_2, \ldots ]$ with $a_0$ a non-negative integer  and  $a_1, a_2, \ldots $  positive integers
We write $x=[a_1, a_2, \ldots ]$ with  $a_1, a_2, \ldots $  positive integers
to denote the  continued fraction expansion of a real number $x$ in the unit interval.
%Here we  only deal with real numbers in the unit interval.
%For brevity we  write $[a_1,a_2\ldots ]$ instead of  $[0; a_1, a_2, \ldots]$.
The functions $p_{n}(x)$ and $q_{n}(x)$, called the  convergents of $x$,  
are defined recursively  as follows. 

If $x=[a_1, a_2, \ldots]$,
\begin{gather*}
p_{-1}(x)=q_0(x)=1
\\
p_0(x)=q_{-1}(x)=0
\end{gather*}
\\
and for $n\geq 1$,
\begin{gather*}
p_n(x)= a_n p_{n-1}(x) + p_{n-2}(x),
\\
q_n(x)= a_n q_{n-1}(x) + q_{n-2}(x).
\end{gather*}
When the real number $x$ is understood from the context
we write  $p_{-1}, q_{-1}, p_0, q_0,\ldots$ instead of  
$p_{-1}(x), q_{-1}(x), p_0(x), q_0(x),\ldots$
Given $a_1,\dotsc,a_n$, we have the equation
\[ 
[a_1,\dotsc,a_n] = \frac{p_n}{q_n}. 
\]
For irrational $x=[a_1, a_2, \ldots]$,  $p_n/q_n$ is  the $n$th approximant to $x$ 
and converges to $x$ as $n$ tends to infinity.
For rational  $x=[a_1, \dotsc, a_n]$,  $x=p_n/q_n$. Observe that 
$(p_n)_{n \geq 1}$ and $(q_n)_{n \geq 1}$ are increasing.

The convergent $q_n$ can be expressed by a   non-recursive formula, as follows.
We write $\amalg$ to denote the disjoint union operation.
We define the set $\cc P$ of subsets of positive integers  as 
\[
\cc P = \big\{C \subset \N :  (\exists D \subset \N) \ C = \amalg_{n \in D} \{n,n+1\} \big\}.
\]
For every pair of positive integers $r,s$ such that $r \leq s$, define
\begin{align*}
\Omega_{r,s} =& \big\{\cc I \subset \{r,\dotsc,s\} : ( \{r, \dotsc , s\}\setminus \cc I )\in \cc P \big\}.
\\
\alpha_{r,s} = &\sum_{\cc I \in \Omega_{r,s}} \prod_{i \in \cc I} a_i.
\end{align*}
And let   
\[
\alpha_{s+1,s}=1\mbox{ and } \alpha_{s+2,s}=0.
\]
The following proposition holds.
\begin{Proposition}\label{form}

Let $x = [a_1,a_2,...] \in (0,1)$.

\begin{enumerate}
\item  For every positive integer $s$,
$q_s(x) = \alpha_{1,s}$.
% \sum_{\cc I \in \Omega_{1,s}} \prod_{i \in \cc I} a_i

\item  
Let $r,r', s, s'$ be positive integers.
If $r \leq r'$ and $s \geq s'$ then 
 $\alpha_{r,s} \leq \alpha_{r',s'}$. 
Equality holds  if and only if $r=r'$ and $s=s'$.

\item Let $r$ and $s$ be positive integers. 
 If $r\leq s$ then $\alpha_{1,s} = \alpha_{1,r} \ \alpha_{r+1,s} + \alpha_{1,r-1} \ \alpha_{r+2,s}$.
%where for any $r \in \mathbb N$, we take $\alpha_{r+1,r}=1$ and $\alpha_{r+2,r}=0$
%so the previous formula makes sense for the reasonable cases $r=s-1$ and $r=s$.

\end{enumerate}
\end{Proposition}

\begin{proof}
Item 1  is true for $s=1$ and it follows by induction from 
$q_n = a_n q_{n-1} + q_{n-2}$.
\linebreak
Items {\em 2} and {\em 3} follow from the  definition of $\alpha_{r,s}$. 
\end{proof}

For a finite continued fraction $[a_1,\dotsc,a_n]$ we consider
the open  interval $I_{[a_1, \dotsc, a_n]}$ containing the numbers whose first $n$ 
digits of their continued fraction expansion  are $a_1,\dotsc,a_n$.
Thus, 
\begin{align*}
I_{[a_1, \dotsc, a_n]} &= ([a_1, \dotsc, a_n], [a_1, \dotsc, a_n +1]), \mbox{ or }
\\
I_{[a_1, \dotsc, a_n]} &= ([a_1, \dotsc, a_n+1], [a_1, \dotsc, a_n]) 
\end{align*}
We say that an interval $I$ is \cf-ary  of {\em order} $n$ if it is 
some $I_{[a_1, \dotsc, a_n]}$.

The  length of a \cf-ary interval   is 
\[
|I_{[a_1,\dotsc,a_n]}| = \frac{1}{q_n(q_n + q_{n-1})}.
\]
If  $a=[a_1,\dotsc,a_r]$, $b=[a_{r+1},\dotsc,a_s]$  and $c=[a_1,\dotsc,a_s]$,
we write $I_{a,b}$ to denote $I_c$.
For   $x=[d_1,\dotsc,d_n]$ we simply write   $q(x)$ to denote $q_n(x)$. 

\begin{Lemma}\label{lemma:relative}
  Let $a=[a_1,\dotsc,a_r]$, $b=[a_{r+1},\dotsc,a_s]$ and $c=[a_1,\dotsc,a_s]$. 
Then,
\begin{enumerate}
\item \qquad  $ q(a) q(b) \ \leq q(c) \ \leq \ 2q(a) q(b)$.
 
 \item \qquad $|I_b| / 2 \ \leq \ {|I_{a,b}|}/{|I_a|}\ \leq \ 2|I_b|$.
\end{enumerate}
% In case  $r \geq 1$ and  $s \geq r+1$  the above inequalities are strict.
\end{Lemma}

\begin{proof}
{\em 1}.\  Using Proposition~\ref{form},
\begin{align*}
q(a) q(b) = & \ \alpha_{1,r} \ \alpha_{r+1,s} \leq \alpha_{1,r} \ \alpha_{r+1,s} + \alpha_{1,r-1} \ \alpha_{r+2,s}
= \alpha_{1,s} = q(c).
\\
q(c) =&\ \alpha_{1,r} \ \alpha_{r+1,s} + \alpha_{1,r-1} \ \alpha_{r+2,s}
  \leq 2  \alpha_{1,r} \ \alpha_{r+1,s} \leq 2q(a)q(b).
\end{align*}
\begin{align*}
\omit\rlap{\hspace*{-1cm}
{\em 2}. \ $|I_{a,b}|^{-1} =\  \alpha_{1,s} \left( \alpha_{1,s} + \alpha_{1,s-1} \right) $}
\\
 =&\  \Big(\alpha_{1,r} \ \alpha_{r+1,s} + \alpha_{1,r-1} \ \alpha_{r+2,s}\Big)
      \Big( \alpha_{1,r} \left(\alpha_{r+1,s} + \alpha_{r+1,s-1}\right) +
      \alpha_{1,r-1} \left(\alpha_{r+2,s} + \alpha_{r+2,s-1} \right)\Big)
\\
 \leq &\  2\alpha_{1,r} \ \alpha_{r+1,s} 
 \Big(
\alpha_{1,r} \left(\alpha_{r+1,s} + \alpha_{r+1,s-1} \right) +
\alpha_{1,r-1} \left(\alpha_{r+1,s} + \alpha_{r+1,s-1} \right)
\Big)
\\
 = &\  2\alpha_{1,r} (\alpha_{1,r} + \alpha_{1,r-1})
     \alpha_{r+1,s}  \left( \alpha_{r+1,s} + \alpha_{r+1,s-1} \right) 
\\=&\
     2 |I_a|^{-1} |I_b|^{-1}.
\end{align*}
\begin{align*}
\omit\rlap{\hspace*{-1cm}
$|I_{a,b}|^{-1} =\  \alpha_{1,s} \left( \alpha_{1,s} + \alpha_{1,s-1} \right) $}
\\                  
= &\  \Big(\alpha_{1,r} \ \alpha_{r+1,s} + \alpha_{1,r-1} \ \alpha_{r+2,s} \Big)
\Big( \alpha_{1,r} \big(\alpha_{r+1,s} + \alpha_{r+1,s-1} \big) +
\alpha_{1,r-1} \big(\alpha_{r+2,s} + \alpha_{r+2,s-1} \big)
\Big)
\\
 \geq &\  \Big(\alpha_{1,r} \ \alpha_{r+1,s} \Big)
\Big( \alpha_{1,r} \big(\alpha_{r+1,s} + \alpha_{r+1,s-1} \big)
\Big) 
\\\geq  &\  \Big(\alpha_{1,r} \ \alpha_{r+1,s} \Big)
\Big( \frac{\alpha_{1,r}+\alpha_{1,r-1}}{2} \Big)  (\alpha_{r+1,s} + \alpha_{r+1,s-1})
\\
=& \ 
\frac{1}{2} |I_a|^{-1} |I_b|^{-1}. 
\end{align*}
\end{proof}

The distribution of $\log q_n$ obeys in the limit a Gaussian law.
It was first proved by  Ibragimov~\cite{Ibragimov1961}.
Then Philipp ~\cite[Satz 3]{Philipp1967} obtained an error term of  $O(n^{-1/5})$, 
which was later improved by Mischyavichyus~\cite{Mischyavichyus1987}  to $O(n^{-1/2} \log n )$.
Morita~\cite[Theorem 8.1]{Morita1994} obtained  the optimal error term 
of  order  $O(n^{-1/2})$;  a different proof of the same bound  was given 
by Vall\'ee~\cite[Th\'eoreme 9]{Vallee1997}.
In the sequel we write $L$ for L\'evy's constant $\pi^2/(12 \log 2)$.%$ \simeq - 1.18656 \ldots$.

\begin{Lemma}[\protect{Morita~\cite[Theorem 8.1]{Morita1994},
Vall\'ee~\cite[Th\'eoreme 9]{Vallee1997}}]
\label{thm:tcl}
The distribution of the random variable $\log q_n(x)$ is asymptotically Gaussian.
There is $K_0$ and  $n_0$ such that for every~$n\geq n_0$,
\[
\left|
 \mbox{Pr} 
 \Big[ x \in (0,1) :  
  -y \leq\frac{  \log q_n(x) -  n L  }{\sigma \sqrt{n}} \leq  y  \Big] -
 \frac{1}{\sqrt{2\pi}} \int_{-y}^y e^{-z^2/2} dz 
\right|<
 \frac{K_0}{\sqrt n},
 \] 
where $\sigma$ is a positive absolute  constant. 
\end{Lemma}

Vall\'ee  in~\cite{Vallee1997} and also in~\cite{FlajoletVallee1998},
obtained  an expression for $\sigma$ using the generalised transfer 
operators $L_s$ for $s>1$  over a suitable space of functions,  
also known as the Ruelle-Mayer operator,
defined by
\[
L_s[f](z) = \sum_{n=1}^\infty \left(\frac{1}{n+z}\right)^s f\!\left(\frac{1}{n+z}\right).
\]
These operators $L_s$ have a simple  dominant positive eigenvalue $\lambda(s)$.
The  expression for $\sigma$  uses the dominant eigenvalue of~$L_2$, 
\[
\sigma^2=\lambda''(2)-\lambda'(2)^2
\]
where $\lambda'$ and $\lambda''$ 
denote the derivative and second derivative of $\lambda$
and 
\begin{align*}
\lambda'(2) &  =- \pi^2/(12 \log 2) \text{ is  Levy's constant with negative sign.}
\\
\lambda''(2) &\text{ is the variance of the law of continuants, known as Hensley's constant.}
\end{align*}
We remark that our use of $\sigma$ occurs just in Lemma \ref{lemma:cfgrande}
and we not require its exact value $\sigma$,
 any upper bound suffices.

\begin{Lemma}\label{lemma:cfgrande}
 There are positive constants $K, c$ and a positive integer $n_1$ 
such that for any \cf-ary interval $I$ 
and any  integer $n\geq n_1$,
 the Lebesgue measure of the union of the \cf-ary subintervals  $J$ of $I$ 
of relative order $n$  such that 
\[
\frac{|I|}{4}e^{-2nL-2c}\leq |J| \leq  2|I|  e^{-2nL+2c}
\]
 is greater than $K|I|/\sqrt{n}$. 
\end{Lemma}

\begin{proof}
Take a positive constant $c > \sigma \sqrt{2 \pi} K_0 $, say
$c = 2 \sigma \sqrt{2 \pi} K_0 $.
For each positive integer $n$ consider the set 
\[
S_n= \{ x\in (0,1): -c \leq \log q_n(x) - nL\leq c\}
\]
By Lemma~\ref{thm:tcl}
there is $K_0$ and $n_0$
such that for every integer  $n\geq n_0$, the Lebesgue measure of $S_n$
is at least 
\[
\frac{1}{\sqrt{2\pi}}\int^{c/(\sigma \sqrt n)}_{-c/(\sigma \sqrt n) } e^{-w^2/2} dw
- \frac{K_0}{\sqrt n}.
\]
But for $n$ sufficiently large so that $e^{-w^2/2} > 1/2$ in the following integration
region,
\[
\frac{1}{\sqrt{2\pi}}\int^{c/(\sigma \sqrt n)}_{-c/(\sigma \sqrt n) } e^{-w^2/2} dw 
> \frac{1}{\sqrt{2 \pi}} \frac{c}{\sigma \sqrt n} = 2\frac{K_0}{\sqrt n},
\]
hence we have that the Lebesgue measure of $S_n$ is at least
$K_0/\sqrt{n}.$

Consider a sequence of digits 
   $a_1, a_ 2, \ldots a_n $
such that 
\[
-c\leq \log q_n(x) - nL\leq c.
\]
This is equivalent to 
\[
e^{-nL-c}\leq \frac{1}{q_n(x)}\leq e^{-nL+c}.
\]
We call $I_n = I_{a_1,\dotsc,a_n}$, whose length is 
$|I_n| = 1/({q_n(q_n+q_{n-1})})$. Clearly
\[
\frac{1}{2q_n^2} \leq |I_n| \leq \frac{1}{q_n^2}.
\]
Fix a \cf-ary interval $I$.
The concatenation of  $a_1 ,\ldots, a_n$ after the digits that define $I$
yields a  \cf-ary subinterval $J$ of $I$. By  Lemma~\ref{lemma:relative},
\[
\frac{1}{4 q^2_n} \leq \frac{|I_n|}{2} \leq  \ \frac{|J|}{|I|} \
\leq 2 |I_n| \leq \frac{2}{q_n^2}.
\]
Thus,
\[
\frac{1}{4}|I|e^{-2nL-2c} \leq  \ |J|\  \leq 2|I|e^{-2nL-2c}.
\]
Then, the set of \cf-ary subintervals $J$ of  $I$ of relative order $n$
such that 
\[  
 \frac{|I|}{4}  e^{-2nL-2c} \leq |J|\leq  2 |I | e^{-2nL+2c}
\]
has Lebesgue measure  at least 
\[
K_0/(2|I|\sqrt{n}).
\]
\end{proof}

\subsection{Discrepancy associated to  continued  fraction expansions}

The Gauss map $T$ is a function from real numbers in the unit interval 
 to real numbers in the unit interval, 
defined by  $T(0)=0$ and  $T(x)= 1/x - \floor{1/x}$.
If $[a_1,a_2, \ldots]$ denotes the continued fraction expansion of $x$, 
 then $T^n(x)=[a_{n+1}, a_{n+2}, \ldots]$ and $a_n=\lfloor 1/T^{n-1}(x)\rfloor$, 
for~$n\geq 1$.
The map $T$ possesses an invariant ergodic  measure, the Gauss measure
$\mu$, which is absolutely continuous with respect to Lebesgue measure, 

\[
\mu({\rm  d} x)= \frac{{\rm  d} x}{(1+x)\log 2}.
\]
%where $\log$ denotes the logarithm in base~$e$.
Then the Gauss measure for a \cf-ary interval  $I_{[a_1, \dotsc,a_n]}$ in the unit interval is, 

\[
\mu(I_{[a_1,\dotsc, a_n]} )= 
\int^{r'/s'}_{r/s} \mu({\rm d}x),
\]
where $r/s$ and $r'/s'$ denote the rational numbers
$[a_1, \dotsc ,a_{n}]$ and
$[a_1, \dotsc ,a_{n}+1]$
ordered such that $r/s<r'/s'$.

We write ${\mathbb I}_I(x)$ to denote the characteristic function of the interval~$I$,
so ${\mathbb I}_I(x)=1$ if $x\in I$ and ${\mathbb I}_I(x)=0$, otherwise.
We say that $[a_1, a_2, \ldots]$ is a normal continued fraction if, 
for every positive integer $k$ and for every 
block  of $k$ positive integers $v_1,\dotsc,v_k$,

\[
\lim_{n\to\infty} \frac{1}{n}\sum_{j=0}^{n-1} {\mathbb I}_{I_{[v_1,\dotsc,v_k]}}(T^j x)  
=\mu(I_{[v_1, \dotsc, v_k]} ).
\]
Equivalently,

\[
\lim_{n\to\infty}\frac{1}{n} \#\{ j: 1\leq j\leq n, a_{j}=v_1, \dotsc, a_{j+k-1}=v_k\}
=\mu(I_{[v_1, \dotsc, v_k]}).
\]

For example, quadratic irrationals do not have a normal continued fraction 
expansion because they are periodic.
 The continued fraction expansion of  $e$, $ [2;1,2,1,1,4,1,1,6,1,1,8,\ldots]$, is not normal either
because it is the concatenation of  the pattern $(1 m 1)$, for all  even $m$ in increasing order.
Applying Birkhoff's Ergodic  Theorem~\cite{Birkhoff1931} we obtain that
% for almost every real in the unit interval with respect to Gauss measure
% the frequency of occurrence of  the block $v_1, \dotsc, v_k $ in $a_1, a_2, \ldots$ 
% in its continued fraction expansion 
% exists and it is  equal to $\mu(I_{[v_1, \dotsc, v_k]} )$.
almost every real in the unit interval has normal continued fraction expansion.
\pagebreak

For a real number $x$ and a block $v$ of $k$ positive integers $v_1,\dotsc, v_k$
the discrepancy of $x$ with respect to $v$ in the first $n$ positions 
of its continued fraction expansion is defined as

\[
\D{\cf-ary}_{v,n}(x)= \Big|\frac{1}{n}\sum_{j=0}^{n-1} {\mathbb I}_{I_{[v_1,\dotsc,v_k]}}(T^j x) - 
\mu\left(I_{[v_1,\dotsc,v_k]}\right)\Big|.
\]
Clearly, a real number $x$  has a normal continued fraction if and only if for every positive integer $k$,
and for every block $v$ of $k$ positive integers, 

\[
\lim_{n\to \infty} \D{\cf-ary}_{v,n}(x)=0.
\]
The following result on large deviations is essentially Kifer, Peres and Weiss' 
Corollary~3.2 in~\cite{KiferPeresWeiss2001}   but conditioning on the first $r$ terms. 

\begin{Lemma} \label{lemma:KPWc}
 Let $I_{[a_1,\dotsc,a_r]}$ be a \cf-ary interval, and let $b$ be a block of $k$ positive integers $b_1,\dotsc,b_k$.
Then for every positive real $\delta$ and for every positive integer $n$,
 \[
\Big|\Big\{ x \in I_{[a_1,\dotsc,a_r]} : 
\ \Big| \frac{1}{n} \sum_{i=0}^{n-1} \mathbb I_{I_{[b_1,\dotsc,b_k]}}(T^{r+i}x) -     \mu(I_{[b_1,\dotsc,b_k]}) \Big| 
 > \delta \Big\} \Big| 
\leq   6 M e^{-\frac{\delta^2n}{2M}} |I_{[a_1\dotsc,a_r]}|, 
\]
 where $M= M(\delta, k) = \Big\lceil k - \log \left({\delta^2}/(2 \log 2) \right ) \Big\rceil$, or any larger number.
\end{Lemma}

\begin{proof}
We write $I_a$ and $I_b$ to denote, respectively, $I_{[a_1,\dotsc,a_r]}$, $I_{[b_1,\dotsc,b_k]}$.
 \\
Define
\begin{align*}
\tau_{b,n,\delta}&= \Big\{ x \in (0,1) : \ \Big|\frac{1}{n} \sum_{i=0}^{n-1} {\mathbb I}_{I_b} (T^i x) - \mu(I_b) \Big| > \delta \Big\}
\\
\tau^a_{b,n,\delta}&=\Big\{ x \in I_a : \ \Big|\frac{1}{n} \sum_{i=0}^{n-1}  {\mathbb I}_{I_b} (T^{r+i}x) - \mu(I_b) \big| > \delta \Big\}
 \end{align*}
The set  $\tau^a_{b,n,\delta}$ is the disjoint union of \cf-ary intervals $I_{ac}$ where $c$ belongs to some
 appropriate set~$\cc Z \subset \N^{k+n-1}$.
The set $\tau_{b,n,\delta}$ is the disjoint union of $I_c$ where $c \in \N^{k+n-1}$ belongs to the same~$\cc Z$.
 Lemma~3.1 and Remark 5.1 both in~\cite{KiferPeresWeiss2001}  establish that 
\[\mu(\tau_{b,n,\delta})
\leq  2 M(\delta, k) e^{-\frac{\delta^2n}{2M(\delta,k)}},
\]
 where 
\[
M(\delta,k) = \min \{ m \in \N : (\log 2) 2^{-m+k} \leq \delta^2/2 \}.
\]
For each $c \in \N^{k+n-1}$, by  Lemma~\ref{lemma:relative}.$2$,
 \[
\frac{|I_{ac}|}{|I_a|} \leq 2 |I_c|.
\]
Adding all these inequalities for $c$ in $\cc Z$ 
and given the fact that
  for all measurable $S$, 
\[
|S|\leq 2 (\log 2) \mu(S) < \frac{3}{2} \mu(S),
\]
we obtain
\[
\frac{|\tau^a_{b, n,\delta}|}{|I_a|} \leq 2 |\tau_{b, n,\delta}| \leq 3 \mu(\tau_{b, n,\delta}) 
 \leq 6 M(\delta, k) e^{-\frac{\delta^2n}{2M(\delta,k)}}.
\]
 Since $x e^{-t/x}$ is increasing in $x$  for $x,t>0$,
 we can replace $M(\delta, k)$ by any larger value $M$.
 \end{proof}

\begin{Lemma}\label{lemma:Dcf}
Let   $x= [a_1, a_2, \dotsc,a_n]$ and   $u= [b_1, \dotsc,b_s]$. 
 Let $v$  be a block of $k$ positive integers
$ v_1, \dotsc, v_k$. Let   $\epsilon$ be a positive real less than $1$.
\begin{enumerate}
\item  If   $\D{\cf-ary}_{v,n}(x)<\epsilon$
and
$\D{\cf-ary}_{v,s}(u)<\epsilon - (k-1)/s$ then $\D{\cf-ary}_{v,n+s}(xu)<\epsilon $.

\item  If   $\D{\cf-ary}_{v,n}(x)<\epsilon$ and
$s/n<\epsilon$ then 
\begin{enumerate}
\item for every $\ell$ such that $1\leq \ell\leq s$, $\D{\cf-ary}_{v,n+\ell}( xu)< 2\epsilon$,

\item   $\D{\cf-ary}_{v,n+s}(ux)< 2\epsilon$.
\end{enumerate}
\end{enumerate}
\end{Lemma}

\begin{proof} 
Let $|[a_1,\dotsc,a_n]|_{v_1,\dotsc ,v_k}=
\left|\{  j:  1\leq j \leq n -k+1 ,  a_j=v_1, \dotsc, a_{j+k-1}= v_k \} \right| $
be the number of occurrences of the block $v_1,\dotsc,v_k$ in the continued fraction
expansion $[a_1,\dotsc,a_n]$.

For part~{\em 1}, assume  $\D{\cf-ary}_{v,n}(x)<\epsilon$ and $\D{cf-ary}_{v,s}(x)<\epsilon - (k-1)/s$.
Then, 
\begin{align*}
|xu|_v&\leq |x|_v + |u|_v) + k-1\\
&\leq (n+s) \mu(I_v) + (n+s)\epsilon - s \frac{k-1}{s} + k-1\\
&=(n+s) (\mu(I_v) + \epsilon).
\\
|xu|_v&\geq |x|_v + |u|_v \\
&\geq (n+s) \mu(I_v) - (n+s)\epsilon \\
&=(n+s) (\mu(I_v) - \epsilon).
\end{align*}
Therefore,   $\D{\cf-ary}_{v,n+s}(xu)<\epsilon $. 

For part {\em 2}, assume  $\D{\cf-ary}_{v,n}(x)<\epsilon$ and   $s/n<\epsilon$. 
Let $\ell$ such that $1\leq \ell \leq s$. 
Then,
\begin{align*}
  \frac{| [a_1,\dotsc,a_n, b_1,\dotsc,b_\ell]|_v}{n+\ell}\; &\leq   \frac{|x|_v}{n+s}  +\frac{s}{n+s}
\\&\leq \ \frac{(\mu(I_v)+\epsilon)n}{n+s} + \frac{\epsilon n}{n+s}
\\&\leq \ (2\epsilon +\mu(I_v))\frac{ n}{n+s} 
\\&\leq \ 2\epsilon \frac{ n}{n+s} +\mu(I_v)
\\&\leq 2\epsilon +\mu(I_v).
\end{align*}
And
\begin{align*}
  \frac{|[a_1,\dotsc,a_n, b_1,\dotsc,b_\ell]|_v}{n+\ell}\; &\geq   \frac{|x|_v}{n+s}\\
  &\geq \frac{(\mu(I_v)-\epsilon)n}{n+s}\\
  &\geq  \mu(I_v)-\epsilon -\mu(I_v)\;\frac{s}{n+s}& \text{by elementary means}\\
  &\geq  \mu(I_v) -\epsilon -\frac{s}{n}\\
  &\geq \mu(I_v)-2\epsilon,&\text{since $\epsilon>s/n$.}
 \end{align*}
We conclude $\D{\cf-ary}_{v,n+\ell}([a_1,\dotsc,a_n,b_1\dotsc,b_\ell])<2\epsilon$. 
\\Item {\em (b)},  $\D{\cf-ary}_{v,n+s}(ux)< 2\epsilon$, is proved similarly.
\end{proof}

\subsection{Discrepancy associated to expansions in a given integer base} \label{sec:integer}

We say that a base is  an integer greater than or equal to $2$,
 a digit in base $b$ is an integer in $\{0, \dotsc, b-1\}$,
and a block in base $b$ is a finite sequence of digits in base $b$.
If $u$ is a block, its length is denoted by~$|u|$.
We define the discrepancy of the first $n$ digits of a  
block $u=a_1,\dotsc, a_{|u|}$  in base~$b$ as
\[
\D{b-ary}_n(u)= \max\left\{ 
\Big|  \frac{1}{n} \#\{j: 1\leq j\leq n,   a_j=s \} -\frac{1}{b} \Big| : s\in \{0,\dots,b-1\}\right\}.
\]
Clearly, a real number $x$ is simply normal to base $b$ if and only if
its expansion in base $b$, $a_1 a_2 \ldots$ is such that
\[
\lim_{n\to\infty} \D{b-ary}_n(a_1\ldots a_n) = 0.
\] 
In   the construction we use  the following   explicit 
bound for the number of blocks of a given length 
having larger discrepancy than a given value.  

\begin{Lemma}[\protect{\cite[Lemma 2.5]{poly}, 
adapted from~\cite[Theorem 148]{HardyWright2008}}]\label{lemma:BHS}
Fix a base $b$ and a block length~$k$.  
For every real $\epsilon$ such that $6/k \leq
  \epsilon \leq 1/b$,  the number of blocks of length $k$ with  
$b$-ary discrepancy  greater than or equal to  $\epsilon$ is at most
$2 \displaystyle{ b^{k+1}  e^{- b \epsilon^2 k/6}}$.
\end{Lemma}

If $v$ and $u$ are blocks, we write $vu$  for their concatenation.

\begin{Lemma}[\protect{\cite[Lemma 3.1]{poly}}]\label{lemma:D}
  Let $u$ and $v$ be blocks in base~$b$ and  let $\epsilon > 0$.

\begin{enumerate}
 \item  If  $\D{b-ary}_{|u|}(u)<\epsilon$ and
$\D{b-ary}_{|v|}(v)<\epsilon$, then $\D{b-ary}_{|uv|}(uv)< \epsilon$.

 \item If $\D{b-ary}_{|v|}(v)<\epsilon$ and  $|u|/|v|<\epsilon$, then 
 \begin{enumerate}
 \item for every $\ell$ less than or equal to $|u|$, 
 $\D{b-ary}_{|v|+\ell}( v u )< 2\epsilon.$
 \item $\D{b-ary}_{|v|+|u|}(uv )< 2\epsilon$. 
\end{enumerate}
\end{enumerate}
\end{Lemma}

%\begin{proof} {\it 1.} It follows easily by considering a proper convex combination.
%
%{\it 2.}
%Let 
%$\occ(a_1\ldots a_n,d)=\#\{  j:  1\leq j \leq n,  a_j=d \} .$
%Fix the base $b$. Let $u$ and $v$ be blocks. 
%Fix $\ell$ and $\epsilon$.
%We write $(vu)_1 \ldots  (vu)_m$ for the block of the first $m$ digits of  $(vu)$.
%\begin{align*}
%  \frac{\occ\bigl( (v u)_1\ldots (vu)_{|v|+\ell}),d\bigr)}{|v|+\ell}\; &\geq   \frac{\occ(v,d)}{|v|+|u|}\\
%  &\geq \frac{(1/b-\epsilon)|v|}{|v|+|u|}\\
%  &\geq 1/b-\epsilon -(1/b)\;\frac{|u|}{|v|+|u|},& \text{\ by elementary means}\\
%  &\geq 1/b-2\epsilon,&\text{\ since $\epsilon>|u|/|v|$.}
%  \end{align*}
%By a similar verification,
%\[
% \frac{\occ\left( (v u)\ldots (vu)_{|v|+\ell}),d\right)}{|v|+\ell}  \leq 1/b   + 2\epsilon.
%\]
%We conclude $\D{b-ary}_{|v|+\ell}( v u )< 2\epsilon.$
%\\
% The proof of  point {\em (b)},  $\D{b-ary}_{|v|+|u|}(uv )< 2\epsilon$,   is similar.
%\end{proof}

\begin{proof} {\it 1.} It follows easily by considering a proper convex combination.

{\it 2.}
Let  $|a_1\ldots a_n|_d=\#\{  j:  1\leq j \leq n,  a_j=d \} .$
Fix the base $b$. Let $u$ and $v$ be blocks. 
Fix $\ell$ and $\epsilon$.
We write $(vu)_1 \ldots  (vu)_m$ for the block of the first $m$ digits of  $(vu)$.
\begin{align*}
  \frac{| (v u)_1\ldots (vu)_{|v|+\ell})|_d}{|v|+\ell}\; &\geq   \frac{|v|_d}{|v|+|u|}\\
  &\geq \frac{(1/b-\epsilon)|v|}{|v|+|u|}\\
  &\geq 1/b-\epsilon -(1/b)\;\frac{|u|}{|v|+|u|},& \text{\ by elementary means}\\
  &\geq 1/b-2\epsilon,&\text{\ since $\epsilon>|u|/|v|$.}
  \end{align*}
By a similar verification,
\[
 \frac{| (v u)\ldots (vu)_{|v|+\ell}|_d}{|v|+\ell}  \leq 1/b   + 2\epsilon.
\]
We conclude $\D{b-ary}_{|v|+\ell}( v u )< 2\epsilon.$
\\
 The proof of  point {\em (b)},  $\D{b-ary}_{|v|+|u|}(uv )< 2\epsilon$,   is similar.
\end{proof}

As usual, for any integer $b$ greater than or equal to $2$, 
 we say that an interval is $b$-ary, if it is of the form
\[\left( 
\frac{a}{b^k}, \frac{a+1}{b^k}\right)
\]
for some positive integer $k$ and and integer $a$ with $0\leq a< b^k$.
In this case we also say that the interval has order $k$.

\section{Proof of Theorem~\ref{thm}} \label{sec:proof}

The following definition is the core of the construction.

\begin{Definition}
 For an integer $t \geq 2$, a $t$-brick is a $t$-uple 
$(\sigma_{\cf},\sigma_2,\dotsc,\sigma_t)$  as follows
\begin{itemize} 

\item[-]  the interval $\sigma_{\cf}$ is \cf-ary;

\item[-]  for  every $d=2,\ldots,t$,
$\sigma_d$ is $d$-ary interval or the union of two 
consecutive $d$-ary intervals of the same order;
 
\item[-] for every $d=2,\ldots,t$,
$\sigma_{\cf} \subset \sigma_d$ 
 and $|\sigma_{\cf}|/|\sigma_d| \geq {1}/({16e^{4c}d})$.
\end{itemize}
\end{Definition}

\begin{Definition}
 If we have a $t$-brick $\vec \sigma=(\sigma_{\cf},\sigma_2,\dotsc,\sigma_t)$ and a $t'$-brick 
 $\vec \tau = (\tau_{\cf},\tau_2,\dotsc,\tau_{t'})$ we say that $\vec \tau$ refines $\vec \sigma$
 if $t' \geq t$, $\tau_{\cf} \subset \sigma_{\cf}$ and $\tau_d \subset \sigma_d$ for $d=2,\dotsc,t$.
 The refinement is said to have discrepancy less than $\epsilon$  if 
\begin{enumerate}
\item[-]  for each $d = 2,\ldots,t$
 the new block $w$ in base $d$ corresponding to the inclusion 
$\tau_d \subset \sigma_d$ has simple discrepancy 
$\D{d-ary}(w)$ less than $\epsilon$.

\item[-] the new \cf-block of  digits $w$
 corresponding to the inclusion $\tau_{\cf} \subset \sigma_{\cf}$
satisfies that for every block $v$ of $t$ digits all less than or equal to $t$,  
$\D{\cf-ary}_{v,|w|}(w)$ is less than 
$\epsilon - (t-1)/|w|$. 
\end{enumerate}
\end{Definition}

We highlight  the following trivial fact about lengths of $d$-ary subintervals
because it will be important  in the proof of Lemma \ref{lemma:main}, our main lemma.

\begin{Proposition}\label{prop:ladrillo}
 Let $d \geq 2$ and $m \in \mathbb N$. Every interval $I$ whose
 Lebesgue measure is less than $d^{-m}$ is contained in a $d$-ary
 interval of order $m$ or in the union of two such intervals.
\end{Proposition}

\begin{Lemma}[main lemma]\label{lemma:main}
Let $t$  be greater than or equal to $2$, 
let $\epsilon$ be a positive real  less than~$1/t$, 
and let  integer $t'$ be equal to $t$ or to $t+1$.
Then,  any $t$-brick $\vec \sigma=(\sigma_{\cf},\sigma_2,\dotsc,\sigma_t)$ 
admits a refinement $\vec \tau = (\tau_{\cf},\tau_2,\dotsc,\tau_{t'})$
with  discrepancy less than $\epsilon$.
% Moreover,  $|\tau_{\cf}|/|\sigma_{\cf}|$ is 
% bounded from below  by a  function of $t$ and $\epsilon$
% and the relative order of $\tau_{\cf}$ is a function~$n_0(t, \epsilon)$. 
The relative order of $\tau_{\cf}$ might be any $n$ greater than
certain $n_0(t, \epsilon)$.
\end{Lemma}

\begin{proof} 
First, assume $t'=t$.
\smallskip

{\em \  Towards the length of $\tau_{\cf}$.} For each $n$ consider $\cc I_n(\sigma_{\cf})$ the \cf-ary subintervals of 
$\sigma_{\cf}$ of relative order $n$.
 Let $K, c, n_1$ be the constants provided by lemma~\ref{lemma:cfgrande}.
 Call $\mathcal J_n(\sigma_{\cf})$ the collection of intervals 
$J \in \cc I_n(\sigma_{\cf})$ such that 
\[
  \frac{1}{4} e^{-2nL-2c} \leq \frac{|J|}{|\sigma_{\cf}|}\leq 2\ e^{-2nL+2c}.
\]
 If $n \geq n_1$, lemma~\ref{lemma:cfgrande} asserts
\[
\frac{  \left|\bigcup_{J \in \mathcal J_n}J  \right|}{|\sigma_{\cf}|} \geq \frac{K}{\sqrt{n}}.
\]
At the end of the proof we will determine a value for $n$ and we will choose 
 $\tau_{\cf}$ as one of these intervals $J$ in $\cc J_n(\sigma_{cf})$. 
\smallskip

{\em Towards the length of  $\tau_d$.}
 For each $n$ and for  each $d=2, \ldots, t$  we call 
\[
m_d = order_d(\tau_d).
\] 
We choose $m_d$ as the largest
integer such that $2e^{-2nL+2c} |\sigma_{cf}| \leq d^{-m_d}$. This choice guarantees
$|J| \leq d^{-m_d}$ for every $J \in \mathcal J_n$ and also
 \[
d^{-m_d-1} \leq 2e^{-2nL+2c} |\sigma_{cf}| = 
8e^{4c} \frac{1}{4} e^{-2nL-2c} |\sigma_{cf}| 
\leq 8e^{4c} |J|.
\]
% Since $\tau_d$ will be a $d$-ary or the union of two consecutive $d$-adics of order $m_d$,
% we conclude for every $J \in \mathcal J_n$.
For each $J \in \mathcal J_n$  we determine $\tau_d^J$ as the 
 $d$-ary interval of order $m_d$ 
 or the union of two consecutive $d$-ary intervals of order $m_d$
that contain $J$ (Proposition~\ref{prop:ladrillo}).  Thus,
\[
\frac{|J|}{|\tau_d^J|} \geq \frac{1}{16e^{4c}d}.
\]
This choice of $m_d$ imposes bounds on $n_d = order(\tau_d) - order(\sigma_d)$
 which only depend on $n$ and $d$, as follows:
\begin{enumerate}

\item Since\ \  $|\sigma_{cf}| \leq |\sigma_d| \leq |\sigma_{cf}| 16e^{4c} d$\ \
then
\[
 \log_d\left(|\sigma_{cf}|/2\right) \leq -order(\sigma_d) \leq \log_d(|\sigma_{cf}| 16e^{4c} d).
\]
Notice  that  $order(\sigma_d)= -\log_d\left(|\sigma_d|/2\right)$ or 
 $order(\sigma_d)= -\log_d(|\sigma_d|)$.

\item And  since \ \
$ 2e^{-2nL+2c} |\sigma_{cf}| \leq d^{-m_d} \leq d \ 2e^{-2nL+2c} |\sigma_{cf}|$\ \ then
\[
 \log_d(2e^{-2nL+2c} |\sigma_{cf}|) \leq -order(\tau_d)=-m_d \leq \log_d(d \ 2e^{-2nL+2c}|\sigma_{cf}|).
\]
\end{enumerate}
We obtain,
\[
 2nL \log_d e - \log_d(4 d e^{2c}) \leq order(\tau_d) - order(\sigma_d) \leq 2nL \log_d e + \log_d(8 e^{2c}d)  
\]
Then, since   $n_d=order(\tau_d) - order(\sigma_d)$,
\[
 2n \frac{L}{\log d} - \frac{2c}{\log d} - 3 \leq n_d \leq 2n \frac{L}{\log d} + \frac{2c}{\log d} + 4.
\]

{\em \ Bad zones.}  
We must pick one interval $J $ in $\mathcal J_n$ in a zone of low discrepancy. 
This is possible because  the measure of the   zones of large discrepancy 
decrease at an exponential rate in~$n$ 
%(Lemmas~\ref{BHSlemma} and~\ref{lemma:KPWc}), 
while the measure of $\mathcal J_n$ decreases only as 
 $K / \sqrt{n}$. % (Lemma~\ref{lemma:cfgrande}).

For each $n$ let   
\[
B^ 0_{d, \sigma_d, m_d, \epsilon}
\] 
be the set of reals in the  $d$-ary subintervals of $\sigma_d$ of order $m_d$
  with $d$-discrepancy greater than~$\epsilon$. And let
\[
B_{d, \sigma_d, m_d, \epsilon}
\]  
be the union of $B^0_{d, \sigma_d, m_d, \epsilon}$ with those numbers
lying in a $d$-ary interval of the same order that is a neighbour to one in
$B^ 0_{d, \sigma_d, m_d, \epsilon}$.

With the conditions $6/n_d \leq \epsilon \leq 1/d$, Lemma~\ref{lemma:BHS}
gives the estimate
\[
 \frac{|B_{d,\sigma_d,m_d,\epsilon}|}{|\sigma_d|} \leq 6d e^{-d \epsilon^2 n_d / 6}.
\]
% For each base $d$, we have to consider $B_{d,\sigma_d,m_d,\epsilon}$
%with  $m_d \geq 2n \frac{L}{\log d} - 7$. 
Since $n_d \geq 2n \frac{L}{\log d} - \frac{2c}{\log d} - 3$,
and $|\sigma_d| \leq 16 e^{4c} d |\sigma_{\cf}|$, we have
 \begin{align*}
 \frac{|B_{d,\sigma_d,m_d,\epsilon}|}{|\sigma_{\cf}| } 
&\leq\
16 e^{4c} d \frac{|B_{d,\sigma_d,m_d,\epsilon}|}{|\sigma_d|} 
\\& \leq \ 96 \ e^{4c} d^2 e^{-d \epsilon^2 n_d / 6} 
\\&\leq A(d) e^{-d \epsilon^2 L n / (3 \log d)}
\end{align*}
where
$A(d)=\ 96 \ e^{4c} d^2    e^{d\epsilon^2( c/(3\log d)+1/2)} $.

For each $n$, let   
\[
\tilde B_{t, \sigma_{\cf}, n, \epsilon}
\]
be the set of reals $x$ in the \cf-ary subintervals of $\sigma_{\cf}$ 
of relative order $n$ 
such that for some block  of length $t$ of digits less than or equal to $t$
the \cf-discrepancy of $x$ is  greater than $\epsilon-(t-1)/n$.
With the condition   $2(t-1)/\epsilon  \leq n$,
it suffices to consider  
\cf-discrepancy   greater than $\epsilon/2$.
Lemma~\ref{lemma:KPWc} gives the estimate, 
 \[
\frac{|\tilde B_{t, \sigma_{\cf}, n, \epsilon}|}{|\sigma_{\cf}|} 
\leq t^t 6Me^{-\frac{(\epsilon/2)^2 n}{2M}},
\]
 where $M = \left\lceil t - \log\left( \frac{(\epsilon/2)^2}{2\log 2}\right) \right\rceil$ or larger.

{\em \ Find $n_0$ large enough.}
Now we choose $n_0$ such that for $n \geq n_0$ 
the bad zones are smaller than the measure of the union of $\cc J_n$.
We need a value of $n$ such that  

\begin{align*}
6M t^t e^{-\frac{(\epsilon/2)^2 n}{2M}} 
&\leq \ \frac{K}{t\sqrt{n}}, 
\qquad \mbox{and for $d=2, \dotsc, t$},
\\
A(d)  e^{-d \epsilon^2 L n / (3 \log d)}
& \leq \ \frac{K}{t \sqrt{n}}.
\end{align*}

Hence, we need to find  solutions to 
 \[
\sqrt{n} e^{-rn} \leq \gamma
 \]
for certain values of $r$  and $\gamma$.
Since for every positive $x$,  it holds that  $x < e^{x/2}$, we have
 \[
\sqrt{n} e^{-r n/2} \leq \frac{1}{r} r\ n\ e^{-rn/2} \ <
 \frac{1}{r}\ e^{r n/2 - r n/2} = \frac{1}{r}.
\]
Thus, we need $n_0$ such that 
\[
e^{-r n/2} \leq \gamma r.
\] 
for each of the needed values $r$ and $\gamma$.
Hence, $n_0$ has to be as large as
\[
 -2/ r \ \log(\gamma r).
\]
for each of the needed values $r$ and $\gamma$.

Letting
\[
\begin{array}{ll}
r^{(1)}=\varepsilon^2/(8M), & \gamma^{(1)}=K/(6 M t^{t+1}),\qquad 
\text{and for $d=2, \dotsc, t$},
\\
r^{(d)}=d \varepsilon^2 L/(3\log d), & \gamma^{(d)}=K/(t\  A(d)).
\end{array}
\]
we have to take  
\[
n_0 = \max \left\{   -2/ r^{(d)} \log\left(\gamma^{(d)} r^{(d)}\right): 1\leq d\leq t\right\}
\cup \left\{\frac{6}{\epsilon},\frac{2(t-1)}{\epsilon},n_1 \right\}
\]
This completes the proof in case $t'=t$.

\medskip

The case $t'=t+1$ follows easily by taking first
a $t$-brick $\vec \tau$ refining $\vec \sigma$ with discrepancy less than $\epsilon$.
Then we only need to take $\tau_{t+1}$ a $(t+1)$-ary interval 
of order $m_{t+1}$, or a union of two consecutive such intervals so that 
$\tau_{cf} \subset \tau_{t+1}$ and $\tau_{t+1}$ not very large. 
For instance, we can take  $m_{t+1}$ to be the maximum such that 
$|\tau_{cf}| \leq (t+1)^{-m_{t+1}}$, so that applying 
again Proposition \ref{prop:ladrillo},
\[
\frac{|\tau_{cf}|}{|\tau_{t+1}|} \geq \frac{1}{2(t+1)}.\qedhere
\]
\end{proof}

\subsection{Algorithm}

 The algorithm constructs a sequence of  $t$-bricks
$\vec{\sigma}_1, \vec{\sigma}_2, \vec{\sigma}_3,\ldots$
for non-decreasing values of~$t$.
The real number defined by the intersection of all the 
 intervals in the sequence is absolutely normal and  
continued fraction normal.

We consider the  block length $t$, the discrepancy value $\epsilon$ 
and the relative order $n$ of the new \cf-ary interval
as functions of the step $s$.
Define for every positive $s$
\begin{align*}
t(s)&=\max(2, \lfloor \sqrt[5]{ \log s} \rfloor),
\\
\epsilon(s)&=\ 1/t(s).
\end{align*}
Clearly $t(s)$ is non-decreasing unbounded 
and $\epsilon(s)$ is  non-increasing and  goes to zero.
Now consider the function  $n_0\big(\epsilon(s), t(s)\big)$ given by Lemma~\ref{lemma:main} 
and notice that 
\[
n_0\big(\epsilon(s), t(s)\big) \mbox{ is in } O\big(t(s)^4 \log (t(s))\big).
\]
Let $n_{\start}$ be the minimum positive integer such that for every positive $s$
\[
\lfloor \log s \rfloor + n_{\start} \geq n_0(\epsilon(s), t(s))
\]
and define
\[
n (s)= \lfloor \log s \rfloor + n_{\start}.
\]
The algorithm is as follows.
\begin{description}
\item[\textnormal {\em Initial step, $s=1$}.]  
Let  $\vec{\sigma}_1=( \sigma_{\cf}, \sigma_2)$, 
with  $\sigma_2=\sigma_{\cf}=(0,1)$.
 
\item[\textnormal {\em Recursive step, $s>1$}.] 
Assume $\vec{\sigma}_{s-1}=(\sigma_{\cf }, \sigma_2,\dotsc,\sigma_{t(s-1)} ) $.
Take $\vec{\sigma}_s =(\tau_{\cf},\tau_2,\dotsc,\tau_{t(s)})$  
the leftmost refinement  of $\vec{\sigma}_{s-1}$  with discrepancy less than 
$\epsilon(s)$ 
and such that  the order of $\tau_{\cf}$ is $n(s)$ plus the order of $\sigma_{\cf}$.
\end{description}

\subsection{Correctness}

The existence of the sequence $\vec{\sigma}_1, \vec{\sigma}_2, \ldots$
is guaranteed by Lemma~\ref{lemma:main}.
We have to prove that the real number $x$ defined by the intersection of all the 
 intervals in the  sequence is absolutely normal and continued fraction normal.

We first show  that $x$ has a normal continued fraction expansion.
Let $v$ be a block of  $m$ integers $v_1,\dotsc,v_m$ 
and let $\tilde \epsilon > 0$.
Choose $s_0$ so that 
$m\leq t(s_0)$, $\max\{v_1, \dotsc,v_m\}\leq t(s_0)$ and  $\epsilon(s_0) \leq \tilde \epsilon / 4$.
At each step $s$ after $s_0$, 
the continued fraction expansion of  $x$ 
is constructed  by appending a block $u_s$ 
such that $|u_s| = n(s)$ and
\[
\D{\cf-ary}_{v,|u_s|}(u_s) \ < \ 
\epsilon(s) - \frac{t(s-1)-1}{|u_s|} \ < \ \epsilon(s) - \frac{m-1}{|u_s|}.
\]
  By Lemma~\ref{lemma:Dcf} (item 1) applied many times, for every $s \geq s_0$:
\[
\D{\cf-ary}_{v,|u_{s_0} \ldots u_s|}(u_{s_0}u_{s_0+1} \ldots u_s) < \epsilon(s_0).
\]
  Next, by Lemma~\ref{lemma:Dcf} (item 2b) there is $s_1$ sufficiently large such
  that for every $s \geq s_1$,
 \[
\D{\cf-ary}_{v,|u_1 \ldots u_s|}(u_1 \ldots u_s) < 2\epsilon(s_0).
\]  
  Since $n(s)$ grows logarithmically, the inequality
\[
n(s)\leq 2\epsilon(s_0)\sum_{j=1}^{s-1} n(j)
\]
  holds from certain point on.
%   which is easy to check after using $n(j) \geq n_{start}$.
  Hence, by Lemma~\ref{lemma:Dcf} (item 2a), we have for every $s$ sufficiently large and 
for every $\ell$ such that  $|u_1 \ldots u_{s-1}| < \ell \leq |u_1\ldots u_s|$,
  \[
   \D{\cf-ary}_{v,\ell}(u_1 \ldots u_s) < 4 \epsilon(s_0) < \tilde \epsilon.
  \]\nopagebreak
  It follows that  $x$ is  continued fraction normal.
  
The argument to show that $x$ is absolutely normal is very similar.
We pick a base $d$ and show that $x$ is simply normal to base $d$.
Let $\tilde \epsilon > 0$. Choose $s_0$ so that
$t(s_0) \geq d$ and $\epsilon(s_0)\leq \tilde \epsilon / 4$.
At each step $s$ after $s_0$
the expansion of  $x$ in base $d$ was constructed by appending blocks $u_s$ 
such that $\D{d-ary}_{|u_s|}(u_s)< \epsilon(s_0)$.
 Thus, by Lemma~\ref{lemma:D} (item 1) for any $s > s_0$,
\[
\D{d-ary}_{|u_{s_0} \ldots u_s|}(u_{s_0} \ldots u_s) < \epsilon(s_0). 
\]
  Applying Lemma~\ref{lemma:D} (item 2a), we obtain $s_1$ such that 
  for any $s > s_1$
\[
\D{d-ary}_{|u_1 \ldots u_s|}(u_1 \ldots u_s) < 2\epsilon(s_0).
\]
Call $n_d(j)$ the relative order of the $d$-interval of $\vec \sigma_{n(j)}$
with respect to the $d$-interval of $\vec \sigma_{n(j-1)}$. 
The inequalities 
\[ 
2n(j) \frac{L}{\log d} - \frac{2c}{\log d} - 3 
\leq n_d(j) \leq 2n(j) \frac{L}{\log d} + \frac{2c}{\log d} + 4 
\]
provided by the proof of Lemma~\ref{lemma:main}, tell us that
$n_d(j)$ grows logarithmically.

\noindent
Thus, for $s$ sufficiently large we have
 \[
 n_d(s) \leq 2\epsilon(s_0) \sum_{j=1}^{s-1}  n_d(j).
 \]
 By Lemma~\ref{lemma:D} (item 2b) we conclude that for $s$ sufficiently large
 and $|u_1 \ldots  u_{s-1}| \leq \ell \leq |u_1 \ldots  u_s|$, 
\[
\D{d-ary}_{\ell}(u_1 \ldots  u_s)
<\ 4 \epsilon(s_0) < \tilde \epsilon.
\]
So, $x$ is simply normal to base $d$ for every $d \geq 2$.

\subsection{Computational complexity}

We analyse the computational complexity of the algorithm described in the previous section
by counting  the number of mathematical operations  required to output the first $k$ digits of 
the continued fraction expansion of the computed number.
We would obtain an equivalent outcome 
if we counted the number of mathematical operations  required to output 
the first $k$ digits of the expansion of the computed number in some prescribed base.
Here we  do not count how many elementary operations are implied by each of the 
 mathematical operations,
which means that we neglect the computational cost of performing arithmetical 
operations with  arbitrary precision.
\smallskip

{\em Memory assumptions at step $s$.}
Let $N(s)=\sum_{i=1}^s n(i)$.
At the beginning of step $s$ the current $t$-brick is 
 $\vec{\sigma}_{s-1}=(\sigma_{\cf}, \sigma_2, \dotsc, \sigma_{t(s-1)})$.
Let  $z_{\cf}$ be the left endpoint of  $\sigma_{\cf}$
and let   $[a_1, \dotsc, a_{N(s-1)}]$ be is continued fraction expansion.
Let $z_d$ be the left endpoint of $\sigma_d$.
We assume that at step $s$ the algorithm 
has direct access to the  following values:
\begin{enumerate}  
\item the approximant  $p_{N(s-1)}/q_{N(s-1)}$  of $[a_1, \dotsc, a_{N(s-1)}]$.
In this way, the algorithm has access to the value  $z_{\cf}$,

\item the value  $|\sigma_{\cf}|$,

\item  the values $z_{\cf} - z_d$, for $d=2, \dotsc, t(s)$
\end{enumerate}
We do not require direct access to any other values.
\smallskip

{\em How to pick a $t$-brick in the good zone at step $s$.}
Lemma~\ref{lemma:main} ensures the existence of the wanted $t$-brick .
To  effectively find it we proceed as follows.
Divide the interval $\sigma_{\cf}$ into 
\[
\lfloor 4 e^{2n(s)L + 2c} \rfloor + 1
\]
 equal intervals. 
Notice that every interval contained in $\sigma_{\cf}$
 whose length is at least 
\[
\frac{1}{4}e^{-2n(s) L-2c} |\sigma_{\cf}|
\]
 will contain as interior point an endpoint of these equal intervals.
 For each endpoint determine if it 
 belongs to a  \cf-ary subinterval $J_{\cf}$ of $\sigma_{\cf}$ of relative order
 $n(s)$ whose length  is between 
\[
 \frac{1}{4}e^{-2n(s) L-2c} |\sigma_{\cf}| \ \mbox{ and } \ 2e^{2n(s) L+2c} |\sigma_{\cf}|.
\] 
In case it does, determine if the   
corresponding $t$-brick $\vec{J}=(J_{\cf}, J_2, \dotsc, J_{t(s)})$ 
refines $\vec{\sigma}_{s-1}$ with  discrepancy less than $\epsilon$.
For this we need to determine the blocks 
that lead from the intervals  $\sigma_{\cf}, \sigma_2, \ldots, \sigma_{t(s-1)}$
to the intervals $J_{\cf}, J_2, \dotsc, J_{t(s-1)}$.
Thus, given $\sigma_{cf}$,  we need
to determine the block of $n(s)$ many digits that leads to $J_{\cf}$.
And for $d=2, \dotsc,t(s-1)$,  given $\sigma_d$, we need to determine
the block of $n_d=order(J_d )-order(\sigma_d)$ many digits in base $d$ that leads to $J_d$.
% If the discrepancy of all these blocks is less than~$\epsilon$ then  
% we conclude that $\vec{J}$ refines $\vec{\sigma}_{s-1}$ with  discrepancy less than~$\epsilon$,

\smallskip

{\em An upper bound for the number of  mathematical operations at step $s$.}
In the worst case, to find a wanted $t$-brick we have to inspect all the candidate endpoints.
Since $n(s)=\lfloor \log s\rfloor + n_{\start}$,
the total number $T$ of candidate endpoints is
\[
\left\lfloor 4 e^{2 (\lfloor \log s\rfloor + n_{\start})  L+2c} \right\rfloor.
\]
Thus, the number of endpoints is in
\[
O\left(s^{2L}\right).
\]
Let  $e_0, \ldots e_{T-1}$  be these endpoints. 
We  write each endpoint $e_j$, for $ j=0, \dotsc, T-1$, as
\[
e_j =z_{\cf} +   |\sigma_{\cf}|  \ j/T.
\]
%Notice that, since $\sigma_{\cf}$ is \cf-ary, its endpoints as well as its length  are rational numbers. 
Let  $u, v$ be integers such that $  |\sigma_{\cf}| \ j/T=u/v$.
Then the continued fraction expansion of   $e_j$ can be written as
$[a_1, \dotsc, a_{N(s-1)}]$ %the digits of the continued fraction associated to  $\sigma_{\cf}$
concatenated with  the continued fraction expansion  of~$u/v$.
We only need $n(s)$ many digits of continued fraction 
expansion of~$u/v$ that we can obtain by running 
the Euclidean algorithm on  $(u,v)$ for $n(s)$ iterations.
This  gives a number of mathematical operations in
\[
O(n(s)).
\]
Let $J_{\cf}$ be the  \cf-ary subinterval of $\sigma_{\cf}$ of relative order $n(s)$.
The computation of its length requires computing the convergents 
$q_{N(s-1)+1}, \dotsc, q_{N(s-1) +n(s)}$. Thus checking that
the length is suitable 
requires a number of mathematical operations in
\[
O(n(s)).
\]
Now we  write each endpoint $e_j$, for $ j=0, \dotsc, T-1$, as
\[
e_j = (z_{\cf} - z_d) + z_d +  |\sigma_{\cf}| \ j/T.
\]
Then, the base-$d$ expansion of $e_j$ 
consists of the base-$d$ expansion of $z_d$
followed by the  base $d$-expansion of 
$(z_{\cf} - z_d) +   |\sigma_{\cf}| \ j/T$.
By the proof of Lemma  \ref{lemma:main}, for each base $d$,
 we just need  $n_d$  many digits of this expansion
and $n_d$  is $ O(n(s))$.
The conversion of the rational value  $(z_{\cf} - z_d) +   |\sigma_{\cf}| \ j/T$
to base $d$ can be done by a constant number of mathematical operations.

Finally, we need to  check if the discrepancy  of each of the $t$ blocks witnessed by $e_j$
 is less than~$\epsilon(s)$. This can be done by a number of comparisons that 
is linear in the length of the block plus a constant number of operations, hence in 
\[
O(n(s)).
\]
We conclude that at step $s$ in the worst case the number of required mathematical operations to 
choose $\vec{\sigma}_s$ can be bounded as
\[
O\left(T  \  \big(n(s) + n(s)+  t(s)\ constant + t(s)\ n(s) \big)\right).
\]
Since $T$ is in $O(s^{2L})$,  $n(s)$ is in  $O(\log( s))$, and $t(s)$ is in $O(\log^{1/5}(s))$,
the total number of mathematical operations at step $s$ is
\[
O\left(s^{2L}  \log^{6/5}(s)\right).
\]

{\em Number of mathematical operations to compute the first $k$ digits.}
After the first  $k$ steps the 
number of digits  of the continued fraction expansion of $x$
obtained by the  algorithm is $N(k)$, which is greater than $k$.
While the  number of mathematical operations performed by the algorithm 
 is in the order of
\[
\sum_{s=1}^{k} s^{2L}  \log^{6/5}(s) \text{ which is less than } k^{2L+1}\log^{6/5} (k),
\]
and this last expesion is in  $ O(k^4)$. 
This completes the proof of Theorem~\ref{thm}.
\bigskip
\bigskip

\noindent
{\bf Acknowledgments.}
We are thankful to Brigitte Vall\'ee  for pointing  to us 
the optimal version of the Central Limit Theorem that establishes that  the 
distribution of the   logarithm of the convergents 
of finite continued fractions is asymptotically Gaussian,  as well as 
 its implication on the length of  \cf-intervals.
Both authors are supported by grant PICT 2014-3260 by Agencia Nacional de 
Promoci\'on Cient\'ifica y Tecnol\'ogica, Argentina.
Becher is  a member of Laboratoire International Associ\'e INFINIS, 
Universit\'e Paris Diderot-CNRS / Universidad de Buenos Aires-CONICET).
Becher worked towards this paper at the  
Erwin Schr\"odinger International Institute for Mathematics and Physics, Austria.
\bigskip

\bibliographystyle{plain}
\bibliography{cf-abs-normal}

\begin{thebibliography}{10}

\bibitem{Adler1985}
R.~Adler, M.~Keane, and M.~Smorodinsky.
\newblock A construction of a normal number for the continued fraction
  transformation.
\newblock {\em Journal of Number Theory}, 13(1):95 -- 105, 1981.

\bibitem{poly}
V.~Becher, P.A. Heiber, and T.~Slaman.
\newblock A polynomial-time algorithm for computing absolutely normal numbers.
\newblock {\em Information and Computation}, 232:1--9, 2013.

\bibitem{Birkhoff1931}
G.D. Birkhoff.
\newblock Proof of the {E}rgodic {T}heorem.
\newblock {\em Proc. Nat. Acad. Sci.}, 17:656--660, 1931.

\bibitem{Borel1909}
\'{E}. Borel.
\newblock Les probabilit\'{e}s d'enombrables et leurs applications
  arithm\'{e}tiques.
\newblock {\em Supplemento di Rendiconti del Circolo Matematico di Palermo},
  27:247--271, 1909.

\bibitem{Bugeaud2012}
Y.~Bugeaud.
\newblock {\em Distribution Modulo One and {D}iophantine Approximation}.
\newblock Number 193 in Cambridge Tracts in Mathematics. Cambridge University
  Press, Cambridge, UK, 2012.

\bibitem{FlajoletVallee1998}
P.~Flajolet and B.~Vall\'ee.
\newblock Continued fraction algorithms, functional operators, and structure
  constants.
\newblock {\em Theoretical Computer Science}, 194(1):1 -- 34, 1998.
\newblock Theorem 1.

\bibitem{HardyWright2008}
G.~H. Hardy and E.~M. Wright.
\newblock {\em An introduction to the theory of numbers}.
\newblock Oxford University Press, Oxford, sixth edition, 2008.

\bibitem{Ibragimov1961}
I.~Ibargimov.
\newblock A theorem from the metric theory of continued fractions.
\newblock {\em Vestnik Leningrad. Univ.}, 16(1):13--24, 1961.

\bibitem{KiferPeresWeiss2001}
Y.~Kifer, Y.~Peres, and B~Weiss.
\newblock A dimension gap for continued fractions with independent digits.
\newblock {\em Israel Journal of Mathematics}, 124(1):61--76, 2001.

\bibitem{MadritschScheererTichy2017}
M.~Madritsch, A.-M.Scheerer, and R.~Tichy.
\newblock Computable absolutely nrmal pisot numbers.
\newblock arXiv:1610.06388v1, 2017.

\bibitem{Madritsch2016}
M.G. Madritsch and B.~Mance.
\newblock Construction of $\mu$-normal sequences.
\newblock {\em Monatshefte f{\"u}r Mathematik}, 179(2):259--280, 2016.

\bibitem{Mischyavichyus1987}
G.A. Mischyavichyus.
\newblock Estimate of the remainder in the limit theorem for the denominators
  of continued fractions.
\newblock {\em Litovskii Matematiceskii Sbomik}, 21(3):63--74, 1987.

\bibitem{Morita1994}
T.~Morita.
\newblock Local limit theorem and distribution of periodic orbits of
  {L}asota-{Y}orke transformations with infinite {M}arkov partitions.
\newblock {\em J. Math. Sot. Japan}, 46(2):309--343, 1994.
\newblock Corrections in Vol. 47 (I) (1997) 191--192.

\bibitem{Philipp1967}
W.~Philipp.
\newblock Ein zentraler grenzwertsatz mit anwendungen auf die zahlentheorie.
\newblock {\em Wahrscheinlichkeitstheorie verw. Geb.}, 8:185--203, 1967.
\newblock Theorem 3.

\bibitem{Pos1957}
A.~G. Postnikov and I.~I. Pyatetskii.
\newblock A {M}arkov-sequence of symbols and a normal continued fraction.
\newblock {\em Izv. Akad. Nauk SSSR Ser. Mat.}, 21(6):729--746, 1957.

\bibitem{Queffelec2006}
M.~Qu\'effelec.
\newblock Old and new results on normality.
\newblock In {\em Dynamics and Stochastics}, IMS Lecture Notes Monogr. Ser.,
  48, pages 225--236. Institute of Mathematics and Statistics, Beachwood, OH,
  2006.

\bibitem{Scheerer2017b}
A.-M. Scheerer.
\newblock On the continued fraction expansion of absolutely normal numbers.
\newblock arXiv:1701.07979v1, 2017.

\bibitem{Vallee1997}
B.~Vall\'ee.
\newblock Operateurs de {R}uelle-{M}ayer generalises et analyse des algorithmes
  d'{E}uclide et de {G}auss.
\newblock {\em Acta Arithmetica}, LXXXI.2:101--144, 1997.
\newblock values for the constants.

\end{thebibliography}
\bigskip

\begin{minipage}{\textwidth}
Ver\'onica Becher
\\
Departmento de  Computaci\'on,   Facultad de Ciencias Exactas y Naturales
\\
Universidad de Buenos Aires \& ICC CONICET, Argentina.
\\
{\tt vbecher@dc.uba.ar}
\medskip\\
Sergio Yuhjtman
\\
Departmento de  Matem\'atica, Facultad de Ciencias Exactas y Naturales\\
Universidad de Buenos Aires, Argentina.
\\
{\tt  syuhjtma@dm.uba.ar}
\end{minipage}

\end{document}